\documentclass[runningheads]{llncs}
\usepackage[T1]{fontenc}
\usepackage{graphicx}
\usepackage{hyperref}
\usepackage{amsmath}
\usepackage{amssymb}
\usepackage{amsfonts}
\usepackage{multicol}
\usepackage{subfigure}
\usepackage[numbers,sort&compress]{natbib}
\usepackage{booktabs}
\usepackage{longtable,tabu}

\usepackage{enumitem}
\setlist[enumerate]{label={\arabic*.}}

\usepackage{color}

\newcommand{\critset}[2]{\operatorname{crit}_{#1}(#2)}
\newcommand{\numcrit}[2]{\left\vert\operatorname{crit}_{#1}(#2)\right\vert}

\newcommand{\broom}[2]{B_{#1}(#2)}

\usepackage{tikz}
\usetikzlibrary{calc}

\usepackage{tkz-graph}

\newcommand{\squishlist}{
 \begin{list}{$\bullet$}
  { \setlength{\itemsep}{0pt}
     \setlength{\parsep}{3pt}
     \setlength{\topsep}{3pt}
     \setlength{\partopsep}{0pt}
     \setlength{\leftmargin}{2.5em}
     \setlength{\labelwidth}{1em}
     \setlength{\labelsep}{0.5em} } }

\newcommand{\squishlisttwo}{
 \begin{list}{$\triangleright$}
  { \setlength{\itemsep}{0pt}
     \setlength{\parsep}{0pt}
    \setlength{\topsep}{0pt}
    \setlength{\partopsep}{0pt}
    \setlength{\leftmargin}{2em}
    \setlength{\labelwidth}{1.5em}
    \setlength{\labelsep}{0.5em} } }

\newcommand{\squishend}{
  \end{list}  }
  
\newcommand{\kcol}{$k$-\textsc{Colouring}}

\begin{document}
\title{Vertex-critical graphs in subfamilies of $(P_4+\ell P_1)$-free graphs}
\titlerunning{Vertex-critical graphs in subfamilies of $(P_4+\ell P_1)$-free graphs}
\author{Iain Beaton\inst{1} \and
Ben Cameron\inst{2} }
%
%
\institute{Acadia University, Wolfville, NS Canada \and
University of Prince Edward Island, Charlottetown, PE Canada}

%
%

\maketitle              
\begin{abstract}
A graph $G$ is $k$-vertex-critical if $\chi(G)=k$ but $\chi(G-v)<k$ for all $v\in V(G)$. In this paper we make progress on the open problem of the finiteness of $k$-vertex-critical $(P_4+\ell P_1)$-free graphs by showing that there are only finitely many $k$-vertex-critical graphs in the following subfamilies of $(P_4+\ell P_1)$-free graphs for all $k\ge 1$ and $\ell\ge 0$:
\vspace{2mm}
\begin{itemize}
    \item $(P_4+\ell P_1,2P_2)$-free graphs,
    \item $(P_4+\ell P_1,\text{chair})$-free graphs,
    \item $(P_4+\ell P_1,P_5,\text{bull})$-free graphs, and
    \item $(P_4+\ell P_1,P_5,\text{cricket})$-free graphs.
\end{itemize}
\vspace{2mm}
\noindent In fact, all but the first of these are special cases of our general result that there are only finitely many $k$-vertex-critical $(P_4+\ell P_1,\broom{4}{m},\broom{3}{m}^{+})$-free graphs for all $k\ge 1$ and $\ell,m\ge 0$. Here $\broom{n}{m}$ is the graph obtained from a path of order $n$ by identifying one of its leaves with the centre vertex of $K_{1,m}$   and $\broom{n}{m}^{+}$ is the graph obtained by identifying an edge of $K_3$ with the edge of $\broom{n}{m}$ with endpoints of degrees $2$ and $m$, respectively. Our results imply the existence of simple polynomial-time certifying algorithms to decide the $k$-colourability of all graphs in these subfamilies for every fixed $k$. 

We also show that $\chi(G)\le \ell+2$ for all $(P_4+\ell P_1,K_3)$-free graphs and all $\ell\ge 0$, improving the previously known upper bound of $2\ell+2$ that followed from Randerath and Schiermeyer's 2004 result on $(P_t,K_3)$-free graphs. More generally, we provide a $\chi$-bound in $O(\ell^{\omega-1})$ for $(P_4+\ell P_1)$-free graphs which improves the bound of $(2\ell+2)^{\omega-1}$ which followed from Gravier, Hoàng and Maffray in 2003 for $P_{t}$-free graphs.

\end{abstract}
\section{Introduction}
The problem of determining the $k$-colourability of a graph for fixed $k$, which we will denote as \kcol{}, has seen a very active area of research since it was shown to be NP-complete for all $k\ge 3$ as one of Karp's original 21 problems~\cite{Karp1972}. 
While there are many directions one can take in this field, we are following a long line of research on developing polynomial-time algorithms to solve \kcol{} when the input graphs are restricted to be from a certain family of graphs. 
The restricted families of graphs we are interested in are those defined by forbidden induced subgraphs. 
A graph $G$ is $H$-free if no induced subgraph of $G$ is isomorphic to $H$. For graphs $H_1$, $H_2$, $\dots,\ H_m$, a  graph is $(H_1,H_2,\dots H_m)$-free if it is $H_i$-free for each $i\in \{1,\dots,m\}$. 
\kcol{} $H$-free remains NP-complete when $H$ contains an induced claw~\cite{Holyer1981,LevenGail1983} and when $H$ contains an induced cycle~\cite{KaminskiLozin2007}. 
Thus, assuming P$\neq$NP, if a polynomial-time \kcol{} algorithm exists for $H$-free graphs, then $H$ must be a linear forest, that is, a forest where each component is a path. 
When forbidding connected graphs, it is known that \kcol{} $P_t$-free graphs is NP-complete for all $k\ge 5$ and $t\ge 6$ and for $k=4$ and $t\ge 7$~\cite{Huang2016}. 
On the other hand, $4$-\textsc{Colouring} $P_6$-free graphs is polynomial-time solvable~\cite{P6free1,P6free2}, and \kcol{} $P_5$-free graphs is polynomial-time solvable \textit{for all} $k$~\cite{Hoang2010}. 
When forbidding a disconnected graph, it is known that \kcol{} $H$-free graphs remains NP-complete for all $k\ge 5$ when $H$ is $P_5+P_2$\cite{ChudnovskyHajebiSpirkl2024} or $2P_4$~\cite{HajebiLiSpirkl2022}, but is polynomial-time solvable for all $k,\ell\ge 0$ when $H$ is an induced subgraph of $P_5+\ell P_1$~\cite{Couturier2015} or of $\ell P_3$~\cite{ChudnovskyHajebiSpirkl2024}. 

In $H$-free graphs with known polynomial-time \kcol{} algorithms, there is very active research on determining which of these have the stronger property that they contain only finitely many $(k+1)$-vertex-critical graphs. 
This implies that \kcol{} is polynomial-time solvable as one can search the input graph for any of the finitely many $(k+1)$-vertex-critical graphs as an induced subgraph to determine its $k$-colourability, where, if found, a $(k+1)$-vertex-critical induced subgraph can be returned as a certificate for a negative answer (see~\cite{P5banner2019}, for example). 
This is especially helpful when paired with the polynomial-time \kcol{} algorithms for $(P_5+\ell P_1)$-free graphs that also return a $k$-colouring of the graph if one exists. Thus, in any subfamily of $(P_5+\ell P_1)$-free graphs where there are only finitely many $(k+1)$-vertex-critical graphs, there exists a polynomial-time \textit{certifying} \kcol{} algorithm for the subfamily. 
See~\cite{McConnell2011} for a survey on certifying algorithms and strong arguments for their importance. 
Thus, there has been lots of recent attention paid to proving various families of graphs contain only finitely many $k$-vertex-critical graphs for all $k$. For a positive integer $k$ and graphs $H_1$, $H_2$, $\dots,\ H_m$, let $\critset{k}{H_1, H_2, \dots,\ H_m}$ denote the set of all $k$-vertex-critical $(H_1, H_2, \dots,\ H_m)$-free graphs, so $\numcrit{k}{H_1, H_2, \dots,\ H_m}$ will denote the number of $k$-vertex-critical $(H_1, H_2, \dots,\ H_m)$-free graphs.

In the least restrictive setting with only a single forbidden induced subgraph, the most comprehensive result is that $\numcrit{4}{H}<\infty$ if and only if $H$ is an induced subgraph of $P_6$, $2P_3$, or $P_4+\ell P_1$ for some $\ell\ge 0$~\cite{Chud4critical2020} .
Further, it was shown in~\cite{AbuadasCameronHoangSawada2022} that $\numcrit{k}{P_3+\ell P_1}<\infty$ for all $k\ge 1$ and $\ell\ge 0$.
It follows from these results (and some we will cite later) that the only graphs $H$ for which the finiteness of $\numcrit{k}{H}$ is unknown for all $k$ is when $H=P_4+\ell P_1$ for all $\ell\ge 1$.
Thus, determining exactly which graphs $H$ admit $\numcrit{k}{H}<\infty$ for all $k$ is equivalent to answering the open question from~\cite{CameronHoangSawada2022}:
 For which values of $k\ge 5$ and $\ell \ge 1$ is $\numcrit{k}{P_4+\ell P_1}<\infty$?
While there are no known infinite families of $k$-vertex-critical $(P_4+\ell P_1)$-free graphs for any $k$ or $\ell$, progress on resolving this open problem is limited to subfamilies of the smallest open case (i.e., when $\ell=1$);
it is known that $\numcrit{k}{P_4+P_1,H}<\infty$ when $H$ is gem~\cite{AbuadasCameronHoangSawada2022}, paw+$P_1$, or any graph of order $4$~\cite{BeatonCameron2025cogemfreeord4finite}.

The limited progress on vertex-critical $(P_4+\ell P_1)$-free graphs is especially surprising in contrast to the abundance of work on vertex-critical $P_5$-free graphs. Unlike for $(P_4+\ell P_1)$-free graphs, there are known infinite families of $k$-vertex-critical graphs; while $\numcrit{4}{P_5}=12$~\cite{MaffrayMorel2012}, $\numcrit{k}{2P_2,K_3+P_1}=\infty$ for all $k\ge 5$~\cite{Hoang2015} (note that $(2P_2,K_3+P_1)$-free graphs is a proper subfamily of $P_5$-free graphs).
Therefore, attention shifted to considering the finiteness of $k$-vertex-critical graphs in further-restricted subfamilies of $P_5$ free graphs.
In~\cite{KCameron2021}, it was shown that $\numcrit{k}{P_5, H}<\infty$ when $H$ is a graph of order $4$ for all $k$ if and only if $H$ is not $2P_2$ or $K_3+P_1$.
In the same paper, an open problem was posed to determine for which graphs $H$ of order $5$ is $\numcrit{k}{P_5, H}<\infty$ for all $k\ge 1$.
Toward an answer to this question, it has been shown that $\numcrit{k}{P_5, H}<\infty$ for all $k$ when $H$ is 
banner~\cite{Brause2022}, $K_{2,3}$ and $K_{1,4}$~\cite{Kaminski2019}, $P_2 + 3P_1$~\cite{CameronHoangSawada2022}, $P_3 + 2P_1$~\cite{AbuadasCameronHoangSawada2022}, $\overline{P_5}$~\cite{Dhaliwal2017}, $\overline{P_3 + P_2}$ and gem~\cite{CaiGoedgebeurHuang2021} (see also \cite{CameronHoang2023}), dart~\cite{Xiaetal2023}, $K_{1,3} + P_1$ and $\overline{K_3 + 2P_1}$~\cite{xia2024results}, and $W_4$~\cite{P5W4conf}.
Further, $\numcrit{k}{P_5, C_5}=\infty$ for all $k\ge 6$~\cite{CameronHoang2023P5C5}.
Thus, the open problem on $\numcrit{k}{P_5,H}$ when $H$ is of order $5$ only remains open when $H$ is one of the $9$ graphs in Figure~\ref{fig:P5Hfreeopencases}.

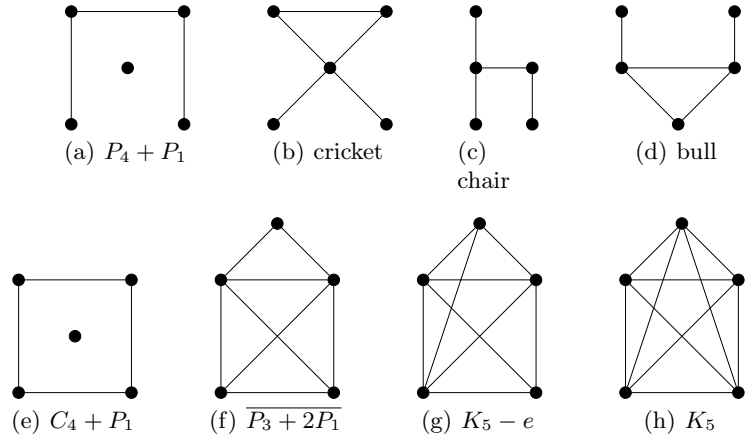
\begin{figure}[h]
\def\c{0.5}
\def\r{1.5}
\centering
\qquad
\subfigure[$P_4+P_1$]{
\scalebox{\c}{
\begin{tikzpicture}
\begin{scope}[every node/.style={circle,fill,draw}]
    \node (u1) at (-1*\r,0*\r) {};
    \node (u2) at (1*\r,0*\r) {};
    \node (u3) at (0*\r,-1*\r) {};
    \node (u4) at (-1*\r,-2*\r) {};
    \node (u5) at (1*\r,-2*\r) {};    
\end{scope}
\begin{scope}
    \path [-] (u1) edge node {} (u2);    
    \path [-] (u5) edge node {} (u2);
    \path [-] (u4) edge node {} (u1);       
\end{scope}
\end{tikzpicture}}
\label{subfig:cogem}
}
\qquad
\subfigure[cricket]{\def\r{1.5}
\centering
\scalebox{\c}{
\begin{tikzpicture}
\begin{scope}[every node/.style={circle,fill,draw}]
    \node (u2) at (-1*\r,0*\r) {};
    \node (u3) at (1*\r,0*\r) {};
    \node (u4) at (0*\r,-1*\r) {};
    \node (w1) at (-1*\r,-2*\r) {};  
    \node (well) at (1*\r,-2*\r) {};    
\end{scope}
\begin{scope}
    \path [-] (u4) edge node {} (w1);
    \path [-] (u4) edge node {} (well);    
    \path [-] (u4) edge node {} (u2);
    \path [-] (u4) edge node {} (u3);
    \path [-] (u2) edge node {} (u3); 
\end{scope}
\end{tikzpicture}}
\label{subfig:cricket}
}
\qquad
\subfigure[chair]{\def\r{1.5}
\centering
\scalebox{\c}{
\begin{tikzpicture}
\begin{scope}[every node/.style={circle,fill,draw}]
    \node (u1) at (0*\r,1*\r) {};
    \node (u2) at (0*\r,0*\r) {};
    \node (u3) at (0*\r,-1*\r) {};
    \node (u4) at (1*\r,0*\r) {};
    \node (u5) at (1*\r,-1*\r) {};    
\end{scope}
\begin{scope}
    \path [-] (u1) edge node {} (u2);   
    \path [-] (u2) edge node {} (u3);  
    \path [-] (u4) edge node {} (u2);  
    \path [-] (u4) edge node {} (u5);   
\end{scope}
\end{tikzpicture}}
\label{subfig:chair}
}
\qquad
\subfigure[bull]{\def\r{1.5}
\centering
\scalebox{\c}{
\begin{tikzpicture}
\begin{scope}[every node/.style={circle,fill,draw}]
    \node (u1) at (-1*\r,1*\r) {};
    \node (u2) at (-1*\r,0*\r) {};
    \node (u3) at (1*\r,0*\r) {};
    \node (u4) at (0*\r,-1*\r) {};
    \node (w1) at (1*\r,1*\r) {};   
\end{scope}

\begin{scope}

    \path [-] (u4) edge node {} (u2);
    \path [-] (u4) edge node {} (u3);
    \path [-] (u2) edge node {} (u3);
    
    \path [-] (u1) edge node {} (u2);
    \path [-] (w1) edge node {} (u3);
    
\end{scope}
\end{tikzpicture}}
\label{subfig:bull}
}
\qquad
\subfigure[$C_4+P_1$]{
\scalebox{\c}{
\begin{tikzpicture}
\begin{scope}[every node/.style={circle,fill,draw}]
    \node (u1) at (-1*\r,0*\r) {};
    \node (u2) at (1*\r,0*\r) {};
    \node (u3) at (0*\r,-1*\r) {};
    \node (u4) at (-1*\r,-2*\r) {};
    \node (u5) at (1*\r,-2*\r) {};    
\end{scope}
\begin{scope}
    \path [-] (u1) edge node {} (u2);    
    \path [-] (u5) edge node {} (u2);
    \path [-] (u4) edge node {} (u1);    
    \path [-] (u4) edge node {} (u5);    
\end{scope}
\end{tikzpicture}}
\label{subfig:C4+P1}
}
\qquad
\subfigure[$\overline{P_3+2P_1}$]{
\scalebox{\c}{
\begin{tikzpicture}
\begin{scope}[every node/.style={circle,fill,draw}]
    \node (u1) at (-1*\r,0*\r) {};
    \node (u2) at (1*\r,0*\r) {};
    \node (u3) at (0*\r,1*\r) {};
    \node (u4) at (-1*\r,-2*\r) {};
    \node (u5) at (1*\r,-2*\r) {};    
\end{scope}
\begin{scope} 
    \path [-] (u1) edge node {} (u3);  
    \path [-] (u2) edge node {} (u3); 
    \path [-] (u1) edge node {} (u2);  
    \path [-] (u1) edge node {} (u4);
    \path [-] (u1) edge node {} (u5);  
    \path [-] (u5) edge node {} (u2);
    \path [-] (u4) edge node {} (u2);   
    \path [-] (u4) edge node {} (u5);    
\end{scope}
\end{tikzpicture}}
\label{subfig:coP3+2P1}
}
\qquad
\subfigure[$K_5-e$]{
\scalebox{\c}{
\begin{tikzpicture}
\begin{scope}[every node/.style={circle,fill,draw}]
    \node (u1) at (-1*\r,0*\r) {};
    \node (u2) at (1*\r,0*\r) {};
    \node (u3) at (0*\r,1*\r) {};
    \node (u4) at (-1*\r,-2*\r) {};
    \node (u5) at (1*\r,-2*\r) {};    
\end{scope}
\begin{scope} 
    \path [-] (u1) edge node {} (u3);  
    \path [-] (u2) edge node {} (u3); 
    \path [-] (u4) edge node {} (u3); 
    \path [-] (u1) edge node {} (u2);  
    \path [-] (u1) edge node {} (u4);
    \path [-] (u1) edge node {} (u5);  
    \path [-] (u5) edge node {} (u2);
    \path [-] (u4) edge node {} (u2);   
    \path [-] (u4) edge node {} (u5);    
\end{scope}
\end{tikzpicture}}
\label{subfig:K5-e}
}
\qquad
\subfigure[$K_5$]{
\scalebox{\c}{
\begin{tikzpicture}
\begin{scope}[every node/.style={circle,fill,draw}]
    \node (u1) at (-1*\r,0*\r) {};
    \node (u2) at (1*\r,0*\r) {};
    \node (u3) at (0*\r,1*\r) {};
    \node (u4) at (-1*\r,-2*\r) {};
    \node (u5) at (1*\r,-2*\r) {};    
\end{scope}
\begin{scope} 
    \path [-] (u1) edge node {} (u3);  
    \path [-] (u2) edge node {} (u3); 
    \path [-] (u1) edge node {} (u2);  
    \path [-] (u1) edge node {} (u4);
    \path [-] (u1) edge node {} (u5);  
    \path [-] (u5) edge node {} (u2);
    \path [-] (u4) edge node {} (u2);   
    \path [-] (u4) edge node {} (u5); 
    \path [-] (u4) edge node {} (u3); 
    \path [-] (u5) edge node {} (u3);    
\end{scope}
\end{tikzpicture}}
\label{subfig:K5}
}
\caption{Graphs $H$ of order $5$ where the finiteness of $k$-vertex-critical $(P_5,H)$-free graphs is unknown.}\label{fig:P5Hfreeopencases}
\end{figure}

Note that the finiteness of $\numcrit{k}{P_5,P_4+P_1}$ is not known for any $k\ge 5$, thus linking interest in the two open problems discussed above. In this paper, we provide the first results on $k$-vertex-critical $(P_4+\ell P_1,H)$-free graphs for $\ell >1$ and, in fact, all of our results hold for all $\ell\ge 0$.
To state our results, we will first require the definitions of two graph families.
Let $\broom{n}{m}$ be the graph obtained from a path of order $n$ by identifying one of its leaves with the centre vertex of $K_{1,m}$ (see Figures~\ref{subfig:Fm} and \ref{subfig:Fm+} for examples with $n=3$ and $4$).  Let $\broom{n}{m}^{+}$ be the graph obtained by identifying an edge of $K_3$ with the edge of $\broom{n}{m}$ with endpoints of degrees $2$ and $m$, respectively (see Figure~\ref{subfig:Bm}).

\begin{center}
\begin{figure}[h]
\def\c{0.7}
\def\r{1.5}
\centering
\qquad
\subfigure[$\broom{3}{m}$]{
    \scalebox{\c}{
        \begin{tikzpicture}
        \begin{scope}[every node/.style={circle,fill,draw}]
            \node[label=above:$u_1$,draw] (u1) at (0*\r,1*\r) {};
            \node[label=above:$u_2$,draw] (u3) at (1*\r,0*\r) {};
            \node[label=above:$u_3$,draw] (u4) at (0*\r,-1*\r) {};
            \node[label=below:$w_1$,draw] (w1) at (-1*\r,-2*\r) {};
            \node[label=below:$w_2$,draw] (w2) at (-0.5*\r,-2*\r) {};  
            \node[label=below:$w_{m}$,draw] (well) at (1*\r,-2*\r) {};    
        \end{scope}
        
        \begin{scope}
        
            \path [-] (u4) edge node {} (w1);
            \path [-] (u4) edge node {} (w2);
            \path [-] (u4) edge node {} (well);
            
            \path [-] (u4) edge node {} (u3);
            
            \path [-] (u1) edge node {} (u3);
            
        \end{scope}
        
        \path (w2) -- node[auto=false]{\ldots} (well);
        
        \end{tikzpicture}
    
    }
    \label{subfig:Fm}
}
\qquad
\subfigure[$\broom{4}{m}$.]{
\centering
    \scalebox{\c}{
        \begin{tikzpicture}
        \begin{scope}[every node/.style={circle,fill,draw}]
            \node[label=above:$u_2$,draw] (u1) at (0*\r,1*\r) {};
            \node[label=above:$u_1$,draw] (u2) at (-1*\r,0*\r) {};
            \node[label=above:$u_3$,draw] (u3) at (1*\r,0*\r) {};
            \node[label=above:$u_4$,draw] (u4) at (0*\r,-1*\r) {};
            \node[label=below:$w_1$,draw] (w1) at (-1*\r,-2*\r) {};
            \node[label=below:$w_2$,draw] (w2) at (-0.5*\r,-2*\r) {};  
            \node[label=below:$w_{m}$,draw] (well) at (1*\r,-2*\r) {};    
        \end{scope}
        
        \begin{scope}
        
            \path [-] (u4) edge node {} (w1);
            \path [-] (u4) edge node {} (w2);
            \path [-] (u4) edge node {} (well);
            
            \path [-] (u4) edge node {} (u3);
            
            \path [-] (u1) edge node {} (u2);
            \path [-] (u1) edge node {} (u3);
            
        \end{scope}
        
        \path (w2) -- node[auto=false]{\ldots} (well);
        
        \end{tikzpicture}
    }
\label{subfig:Fm+}
}
\qquad
\subfigure[$\broom{3}{m}^{+}$.]{
    \scalebox{\c}{
        \begin{tikzpicture}
        \begin{scope}[every node/.style={circle,fill,draw}]
            \node[label=above:$u_2$,draw] (u2) at (-1*\r,0*\r) {};
            \node[label=above:$x$,draw] (u3) at (1*\r,0*\r) {};
            \node[label=above:$u_3$,draw] (u4) at (0*\r,-1*\r) {};
            \node[label=below:$w_1$,draw] (w1) at (-1*\r,-2*\r) {};
            \node[label=below:$w_2$,draw] (w2) at (-0.5*\r,-2*\r) {};  
            \node[label=below:$w_{m}$,draw] (well) at (1*\r,-2*\r) {};   
        
            \node[label=above:$u_1$,draw] (x) at (-2*\r,1*\r) {};
            
        \end{scope}
        
        \begin{scope}
        
            \path [-] (u4) edge node {} (w1);
            \path [-] (u4) edge node {} (w2);
            \path [-] (u4) edge node {} (well);
            
            \path [-] (u4) edge node {} (u2);
            \path [-] (u4) edge node {} (u3);
            
            \path [-] (u2) edge node {} (u3);
            \path [-] (u2) edge node {} (x);

        \end{scope}
        
        \path (w2) -- node[auto=false]{\ldots} (well);
        \end{tikzpicture}
    }
    \label{subfig:Bm}
}
\caption{The forbidden induced graphs in Theorems~\ref{thm:P4UellP1F+mBm}} and Corollary~\ref{thm:P4UellP1Fm}\label{fig:FmFm+Bm}
\end{figure}
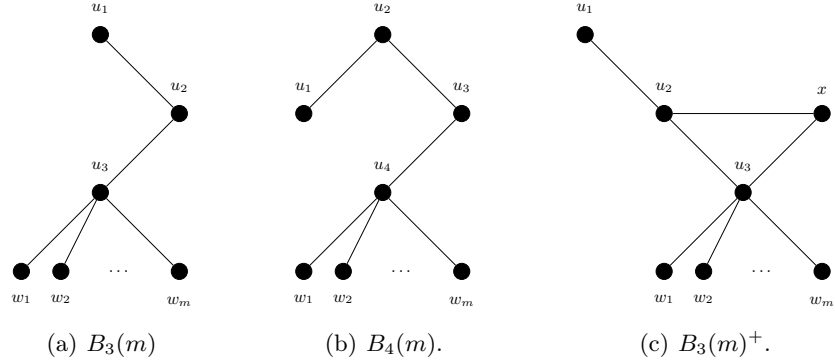
\end{center}

We are now ready to state the two main results of our paper, and discuss some corollaries.

\begin{theorem}\label{thm:P4UellP1F+mBm}
    For all $k\ge 1$ and $\ell,m\ge 0$, $\numcrit{k}{P_4+\ell P_1, \broom{4}{m}, \broom{3}{m}^{+}}<\infty$.
\end{theorem}

\begin{theorem}\label{thm:P4UellP12P2}
    For all $k\ge 1$ and $\ell\ge 0$, $\numcrit{k}{P_4+\ell P_1, 2P_2}<\infty$.
\end{theorem}

Despite its generality, we give a very short and easy to follow proof of Theorem~\ref{thm:P4UellP1F+mBm}. 
We achieve this by a new technical lemma that imposes helpful structure on the non-neighbours of large independent sets of $(P_4+\ell P_1)$-free graphs that we use in conjunction with the aforementioned result that $\numcrit{k}{P_3+\ell P_1}<\infty$ for all $k\ge 1$ and $\ell\ge 0$. 
The generality of Theorem~\ref{thm:P4UellP1F+mBm} leads to many specific corollaries of interest, some of which we now list. 
First, since $\broom{3}{m}$ is an induced subgraph of both $\broom{3}{m}^{+}$ and $\broom{4}{m}$, we get.

\begin{corollary}\label{thm:P4UellP1Fm}
     For all $k\ge 1$ and $\ell,m\ge 0$, $\numcrit{k}{P_4+\ell P_1, \broom{3}{m}}<\infty$.
\end{corollary}

\noindent Since $\broom{3}{2}$ is chair, we get the even more specific corollary that follows.

\begin{corollary}\label{cor:P4UellP1chair}
    For all $k\ge 1$ and $\ell,m\ge 0$, $\numcrit{k}{P_4+\ell P_1, \operatorname{chair}}<\infty$.
\end{corollary}

\noindent Since $P_5$ is $\broom{4}{1}$, bull is $\broom{3}{1}^{+}$ and cricket is the graph obtained from $\broom{3}{2}^{+}$ by deleting the leaf adjacent to the vertex of degree 3, we obtain the following immediate corollaries.

\begin{corollary}\label{cor:bull}
    For all $k\ge 1$, $\numcrit{k}{P_4+\ell P_1,P_5, \operatorname{bull}}<\infty$.
\end{corollary}

\begin{corollary}\label{cor:cricket}
    For all $k\ge 1$, $\numcrit{k}{P_4+\ell P_1,P_5, \operatorname{cricket}}<\infty$.
\end{corollary}

Theorems~\ref{thm:P4UellP1F+mBm} and \ref{thm:P4UellP12P2} provide the best evidence to date that $\numcrit{k}{P_4+\ell P_1, P_5}$ might be finite for all $k$. Further, Corollaries~\ref{cor:P4UellP1chair}, \ref{cor:bull}, and \ref{cor:cricket}  are likely to be helpful for resolving three of the 9 remaining open cases of the open problem stated from~\cite{KCameron2021} stated above, as a now only graphs containing an induced $P_4+\ell P_1$ for any fixed $\ell$ depending only on the chromatic number need to be considered. We note that $\numcrit{k}{P_5,\operatorname{bull}}$ and $\numcrit{k}{P_5,\operatorname{chair}}$ are of particular interest as it is known that each is finite for $k=5$ as shown in \cite{HuangLiXia2023}, and \cite{HuangLi2023}, respectively (and in fact it was further shown that $\numcrit{5}{P_6,\operatorname{bull}}<\infty$~\cite{P6bullfree}).

\subsection{Outline}

The rest of the paper is organized as follows. We first present some preliminary results and terminology in Section~\ref{sec:prelims} that will be used throughout our proofs of our main results. 
In Section~\ref{sec:P5bull}, we prove Theorem~\ref{thm:P4UellP1F+mBm}, and
in Section~\ref{sec:2P2}, we prove Theorem~\ref{thm:P4UellP12P2}. 
In Section~\ref{sec:K_k}, where we show that for all $\ell\ge 0$ and $k\ge 3$, every $(P_4+\ell P_1, K_k)$-free graph $G$ has 

    $$\chi(G) \leq \ell^{k-2}+2\ell^{k-3}+3\ell^{k-4}\cdots+(k-2)\ell+(k-1).$$
We conclude in Section~\ref{sec:conclusion} with a discussion on future directions.

\section{Preliminaries}\label{sec:prelims}

    For standard graph theory terminology and notation, we follow~\cite{WestGraphTheoryBook}. We denote to adjacent vertices $u$ and $v$ by $u \sim v$. Two vertices $u$ and $v$ are \emph{twins} if $N(u)=N(v)$ or $N[u]=N[v]$. We say $u$ and $v$ are \emph{false twins} if $N(u)=N(v)$ and \emph{true twins} if $N[u]=N[v]$. Two disjoint sets of vertices $X$ and $Y$ are \emph{complete} to each other if every vertex in $X$ is adjacent to every vertex in $Y$. We say $X$ and $Y$ are \emph{anti-complete} if there are no edges between the two sets. A vertex is \emph{mixed} on a set of vertices $X$ if it is not complete nor anti-complete to $X$, that is, it is adjacent to at least one vertex in $X$ and non-adjacent to at least one other.

    We now state key lemma and theorem that will be applied in our main theorems.

     \begin{lemma}[\cite{KCameron2021}] \label{lem:XY}
		Let $G$ be a $k$-vertex-critical graph. There does not exist two vertex subsets $X, Y \subseteq V(G)$ satisfying all of the following conditions.
		\begin{itemize}
			\item $X$ and $Y$ are anticomplete to each other.
			\item $\chi(G[X])\le\chi(G[Y])$.
			\item Y is complete to $N(X)$.
		\end{itemize}
	\end{lemma}

\begin{theorem}[\cite{AbuadasCameronHoangSawada2022}]\label{thm:finiteP3ellP1freecrit}
There are only finitely many $k$-vertex-critical $(P_3+\ell P_1)$-free graphs for all $k\ge 1$ and $\ell \ge 0$.
\end{theorem}

\section{$(P_4+\ell P_1$, $\broom{4}{m}$, $\broom{3}{m}^{+})$-free}\label{sec:P5bull}

Before we can prove Theorem~\ref{thm:P4UellP1F+mBm}, we require the following technical lemma on the structure of $P_4$-free graphs that we will apply to the non-neighbours of a large independent set. Note that every component in a $P_4$-free graph can be constructed from a single vertex and iteratively choosing a vertex $v$ and adding a copy of $v$ which is either a false twin or true twin to $v$ (see \cite{brandstadt1999graph}, Theorem 11.3.3). 

\begin{lemma}\label{lem:comparablesetsP3nonneighborhood}
    Let $G$ be a connected $P_4$-free graph. If $G$ contains an induced $P_3$ then there exists disjoint sets $X$ and $Y$ which are anticomplete to each other such that $Y$ is complete to $N(X)\neq\emptyset$ and $X$ is complete to $N(Y)$ (and therefore, $N(X)=N(Y)\neq\emptyset$).
\end{lemma}

\begin{proof}
    We will proceed by induction on the order of $G$, a connected $P_4$-free graph with an induced $P_3$. 
    If $G$ has three vertices, then $G \cong P_3$, and $X=\{x\}$ and $Y=\{y\}$ satisfies the conclusion when $x$ and $y$ are the leaves of the induced $P_3$.
    Suppose that our claim is true for all connected $P_4$-free graphs of order $k$ with an induced $P_3$.
    
    Suppose $G$ has order $k+1$.
    As $G$ is $P_4$-free, there exist two twin vertices $u$ and $v$.
    If $u$ and $v$ are false twins, then $X=\{u\}$ and $Y=\{v\}$ satisfies the conclusion.
    So suppose $u$ and $v$ are true twins.
    By induction, there exist disjoint and anticomplete sets $X$ and $Y$ in $G-v$ such that $N_{G-v}(X)\neq\emptyset$ and $X$ is complete to $N_{G-v}(Y)$.
    Let $W=N_{G-v}(X)=N_{G-v}(Y)$.
    If $u \in X$ (resp. $u \in Y$) then $N(v) \subseteq W \cup X$ (resp. $N(v) \subseteq W \cup Y$). 
    Thus $X\cup\{v\}$ and $Y$ (resp. $X$ and $Y \cup \{v\}$) would have the desired properties in $G$.
    
    So suppose $u \notin X$ and $u \notin Y$.
    If $u \in W$ then $v$ is complete to $X$ and $Y$ and thus $X$ and $Y$ have the desired property in $G$.
    If $u \notin W$ then $v$ is anticomplete to $X$ and $Y$ and thus $X$ and $Y$ have the desired property in $G$. This completes the proof.
\end{proof}

With this lemma in hand, we can now proceed with the proof of our first main theorem.

\begin{proof}[Proof of Theorem~\ref{thm:P4UellP1F+mBm}]

    If $\ell=0$ or $m=0$, then we are only dealing with $P_4$-free graphs. 
    It follows from the Strong Perfect Graph Theorem \cite{Chudnovsky2006} that $P_4$-free graphs are perfect.
    Moreover, there is exactly one perfect $k$-vertex-critical graph, $K_k$, so the result is immediate.
    So, let $k,\ell$ and $m$ be given positive integers and let $G$ be a $k$-vertex-critical $(P_4+\ell P_1$, $\broom{4}{m}$, $\broom{3}{m}^{+})$-free graph. 
    We will show the equivalent statement that $G$ necessarily has finite order. 
    Let $c=\ell+m+1$.
    If $\alpha(G)<c$, then Ramsey's Theorem \cite{Ramsey} implies that there exists an $n$ such that all graphs with at least $n$ vertices contain an independent set of order $c$ or a clique of order $k+1$.
    Therefore there are only finitely many $k$-colorable graphs with $\alpha(G)<c$.
    So we may suppose $\alpha(G)\ge c$.
    
    By Theorem~\ref{thm:finiteP3ellP1freecrit} there are only finitely many $k$-vertex-critical $(P_3+c P_1)$-free graphs.
    Thus, if $G$ does not have finite order, there must exist an independent set $S$ with $|S|=c$ such that $G-N[S]$ contains an induced $P_3$.
    As $G$ is ($P_4+\ell P_1$)-free, then $G-N[S]$ is $P_4$-free.
    Let $C$ be a component of $G-N[S]$ that contains an induced $P_3$.
    
    By Lemma~\ref{lem:comparablesetsP3nonneighborhood}, there must be anticomplete sets $X,Y\subseteq G-N[S]$ such that $Y$ is complete to $N_C(X)\neq \emptyset$ and $X$ is complete to $N_C(Y)$.
    Without loss of generality suppose $\chi(X) \leq \chi(Y)$.
    By Lemma~\ref{lem:XY}, we must have $x\in X$ and $y\in Y$ such that there are vertices $x',z$ such that $x'\in N(x)\setminus N(y)$ and $z\in N_C(x)\cap N_C(y)$. 
    By our assumption, we know that $x'\in N(S)$.
    Without loss of generality, let $s_1$ be a neighbour of $x'$ in $S$.
    See Figure \ref{fig:indGsplit} for an illustration of $G$. Note the dashed lines are potential adjacencies of $x'$.

    \begin{center}  
    \begin{figure}[h] \label{fig:indGsplit}
    \begin{center}
    
    \def\c{0.75}
        \scalebox{\c}{
        \begin{tikzpicture}
            \draw (1.5,.25) ellipse (2.5cm and 0.75cm);
            \node[label=right:\Large $S$]  at (4,0.25) {};

           \draw (1.5,-1.5) ellipse (2.5cm and 0.75cm);
            \node[label=right:\Large $N(S)$]  at (4,-1.5) {};

            \draw (1.5,-3.75) ellipse (3cm and 1cm);
           \node[label=right:{\Large $G-N[S]$}]  at (4.5,-3.75) {};
           
        \begin{scope}[every node/.style={circle,fill,draw}]
            \node[label=above:$s_1$,draw] (s1) at (0,0) {};
            \node[label=above:$s_2$,draw] (s2) at (1,0) {};
            \node[label=above:$s_c$,draw] (sc) at (3,0) {};
            \node[label=left:$x'$,draw] (x1) at (0.25,-1.5) {};
            \node[label=below:$x$,draw] (x) at (0.5,-3.25) {};  
            \node[label=below:$y$,draw] (y) at (2.5,-3.25) {};    
            \node[label=below:$z$,draw] (z) at (1.5, -4) {}; 
        \end{scope}
        
        \begin{scope}
        
            \path [-] (s1) edge node {} (x1);
            \path [-] (x1) edge node {} (x);
            \path [-] (x) edge node {} (z);     
            \path [-] (z) edge node {} (y);
            
        \end{scope}
        
        \path (s2) -- node[auto=false]{\ldots} (sc);
        \path [style=dashed](x1) edge node {} (s2);
        \path [style=dashed](x1) edge node {} (sc);
        \path [style=dashed](x1) edge node {} (z);
        
        \end{tikzpicture}}
    \caption{Illustration of $G$ in proof of Theorem \ref{thm:P4UellP1F+mBm}}
    \end{center}
    \end{figure}
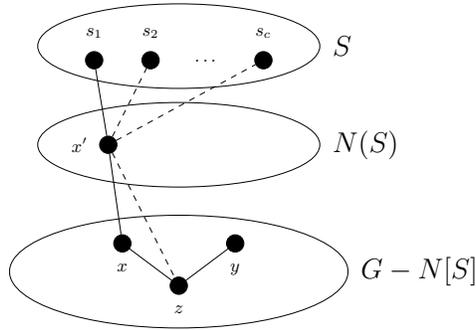
    \end{center}

    Let $S'$ be the set $S\setminus N(x')$. 
    If $|S'|\ge \ell$, then $\{z,x,x',s_1\}\cup S'$ induces a graph containing an induced $P_4+\ell P_1$ if $z\nsim x'$ 
    and
    $\{y,z,x',s_1\}\cup S'$ induces a graph containing an induced $P_4+\ell P_1$ if $z\sim x'$.
    Therefore, $|S'|\le \ell-1$, and $|N(x')\cap S|\ge c-\ell+1=m$.
    Now, the union of $\{y,z,x,x'\}$ and any subset of $S$ containing $m$ neighbours of $x'$ induces $\broom{4}{m}$ if $z\nsim x'$ and $\broom{3}{m}^{+}$ if $z\sim x'$.

    Since we reach a contradiction in both cases, we obtain the desired result.
\end{proof}

\section{$(P_4+\ell P_1$, $2P_2)$-free}\label{sec:2P2}

In this section, part of our proof (Claim~\ref{cla:Uorderatleastk}) is an adaptation of part of the  proof in~\cite{BeatonCameron2025cogemfreeord4finite} that $\numcrit{k}{P_4+P_1,P_5, P_3+cP_2}<\infty$.
The majority of our proof uses new techniques that allow us to get our result for $(P_4+\ell P_1)$-free graphs for all $\ell >1$ (as opposed to $\ell=1$ in~\cite{BeatonCameron2025cogemfreeord4finite}). 

\begin{proof}[Proof of Theorem~\ref{thm:P4UellP12P2}]
If $\ell=0$, then the graph is $P_4$-free and the result is immediate (see comment at the beginning of the Proof of Theorem~\ref{thm:P4UellP1F+mBm}).
Therefore, let $k\ge 1$ and $\ell\ge 1$ be given positive integers and let $G$ be a $k$-vertex-critical $(P_4+\ell P_1$, $2P_2)$-free graph. 
We will prove that $G$ is $(P_3+cP_1)$-free for 

$$c=\binom{k\ell}{\lfloor\frac{k\ell}{2}\rfloor}.$$ 

Suppose, by way of contradiction, that $G$ contains an induced $P_3+cP_1$ and let $P=\{p_1,p_2,p_3\}$ and $S=\{s_1,s_2,\ldots,s_{c}\}$ such that $P\cup S$ induces a $P_3+cP_1$ where the $P_3$ is in order of the indices.  
Let $M$ be the set of all vertices in $V(G)$ that are mixed on $S$. 
Partition $M$ into sets such that all vertices with the exact same neighbours and non-neighbours in $S$ belong to the same set of the partition. 
Let $U$ be a subset of $M$ defined by taking exactly one vertex from each of the sets in the partition. The following claim is a consequence of Sperner's Theorem~\cite{Sperner} and a series of claims in \cite{BeatonCameron2025cogemfreeord4finite} which only depend on $G$ being $k$-vertex-critical and containing an induced $P_3+cP_1$.
To keep this paper self-contained, we will include its proof here.



\begin{claim}[\cite{BeatonCameron2025cogemfreeord4finite}, Claim 3.4]\label{cla:Uorderatleastk}
$|U|\ge k\ell$.
\end{claim}

 \begin{proof}[Proof of Claim~\ref{cla:Uorderatleastk}]We will show that $A=\{N(s)\cap U: s\in S\}$ is an antichain in the partial order $(\mathcal{P}(U),\subseteq)$. Then

 $$|A|=c=\binom{k\ell}{\lfloor\frac{k\ell}{2}\rfloor},$$ 
 
 \noindent and follows from Sperner's Theorem~\cite{Sperner} that $|U|\ge k\ell$.

 Suppose $A$ is not an antichain. 
 Then there exists $s_1, s_2 \in S$ such that $N(s_1)\cap U \subseteq N(s_2)\cap U$.
 As $S$ is an independent set, then, by Lemma \ref{lem:XY}, there must be a vertex $u_1 \sim s_1$ and $u_1 \not\sim s_2$.
 So $u_1$ is mixed on $S$ and hence $u_1 \in M$.
 By the construction of $U$, there is a vertex in $u_1' \in U$ with the same neighbors and non-neighbors as $u_1$ in $S$.
 Thus $u_1' \sim s_1$ and $u_1' \not\sim s_2$.
 Hence $u_1' \in N(s_1)\cap U$ but $u_1' \notin N(s_1)\cap U$ so $N(s_1)\cap U \not\subseteq N(s_2)\cap U$ which is a contradiction.
  
\end{proof}


\begin{claim}\label{cla:indnumofU}
$\alpha(G[U]) \le \ell$. 
\end{claim}
\begin{proof}[Proof of Claim~\ref{cla:indnumofU}]
    Suppose by way of contradiction that $U$ contains an independent set of order at least $\ell+1$ and let $U'=\{u_1,u_2,\dots u_{\ell+1}\}$ be an independent subset of $U$.
    Moreover, without loss of generality, let 

    $$|N(u_1)\cap S| \leq |N(u_2)\cap S| \leq \cdots \leq |N(u_{\ell+1})\cap S|.$$

    \noindent We claim that $N(u_{i})\cap S \subset N(u_{j})\cap S$ for each $1 \leq i< j\leq \ell+1 $.
    To show a contradiction, suppose not.
    By our definition of $U$ and $U'$, clearly $N(u_{i})\cap S \neq N(u_{j})\cap S$ and $N(u_{j})\cap S \not\subset N(u_{i})\cap S$.
    So there exists $s_i, s_j \in S$ such that $s_i \in (N(u_{i})\cap S) \setminus (N(u_{j})\cap S)$ and  $s_j \in (N(u_{j})\cap S) \setminus (N(u_{i})\cap S)$.
    Furthermore $U'$ and $S$ are each independent sets so $u_i \not\sim u_j$ and $s_i \not\sim s_j$.
    However, $\{u_i, s_i, u_j, s_j\}$ induce a $2P_2$ which is a contradiction.
    Therefore $N(u_{i})\cap S \subset N(u_{j})\cap S$ for each $1 \leq i< j\leq \ell+1 $ and hence

    $$|N(u_1)\cap S| < |N(u_2)\cap S| < \cdots < |N(u_{\ell+1})\cap S|.$$

    By the definition of $U$, each $u_i$ is mixed on $S$.
    Therefore $|N(u_{\ell+2})\cap S| \leq |S|-1$ and hence $1 \leq |N(u_{1})\cap S| < |N(u_{2})\cap S| \leq |S|-\ell$.
    Let $T \subset S$ be a set of $\ell$ vertices not adjacent to $u_2$.
    Note $T$ is an independent set.
    Moreover, no vertex in $T$ is adjacent to $u_1$ or any other vertex in $S$.
    Now let $s_1 \in N(u_{1})\cap S$ and $s_2 \in (N(u_{2})\cap S) \setminus (N(u_{1})\cap S)$.
    As $N(u_{1})\cap S \subset N(u_{2})\cap S$ then $s_1 \in N(u_{2})\cap S$ and hence $\{u_1, s_1, u_2, s_2\}$ induces a $P_4$.
    Moreover $\{u_1, s_1, u_2, s_2\} \cup T$ induces a  $P_4+ \ell P_1$ which is a contradiction.
\end{proof}

\begin{claim}\label{cla:chromaticnumofUatleastk}
$\chi(G[U])\ge k$. 
\end{claim}
\begin{proof}[Proof of Claim~\ref{cla:chromaticnumofUatleastk}]
    We have $\frac{|U|}{\alpha(G[U])}\ge \frac{k\ell}{\ell}=k$ by Claims~\ref{cla:Uorderatleastk} and \ref{cla:chromaticnumofUatleastk}. 
    Further, $\frac{|U|}{\alpha(G[U])}\le \chi(G[U])$ by an elementary bound on the chromatic number. Therefore, we have $\chi(G[U])\ge k$.
\end{proof}

We now complete the proof of the theorem. Claim~\ref{cla:chromaticnumofUatleastk} contradicts $G$ being $k$-vertex-critical since we will have $\chi(G-v)\ge k$ for all $v\in V(G)\setminus U$. So, it must be that $G$ is $(P_3+cP_1)$-free and therefore that there are only finitely many $(P_4+\ell P_1, 2P_2)$-free graphs for all $k\ge 1$ by Theorem~\ref{thm:finiteP3ellP1freecrit}.
\end{proof}

\section{$(P_4+\ell P_1, K_k)$-free}\label{sec:K_k}

In this section, we explore the possibility of removing the restriction of $\broom{4}{m}$ from Theorem~\ref{thm:P4UellP1F+mBm}. In order for $\numcrit{k}{P_4+\ell P_1,\broom{3}{m}^{+}}<\infty$ there must necessarily be only finitely many $k$-vertex-critical $(P_4+\ell P_1, K_3)$-free graphs. We will resolve this for $k\geq \ell+3$ through a colourability result on $(P_4+\ell P_1, K_k)$-free graphs.

Note that it is known that every $(P_t,K_3)$-free graphs is $(t-2)$-colourable~\cite{RanderathSchiermeyer2004}, so it follows that every $(P_4+\ell P_1,K_3)$-free graph is $(2\ell+2)$-colourable. We improve this result with a corollary of the following general result.

\begin{theorem}
    Let $H$ be a finite graph, $\ell \geq 0$ and $k \geq 1$. If every $(H,K_3)$-free graph is $k$-colourable, then every $(H+\ell P_1,K_3)$-free graph is $(k+\ell)$-colourable
\end{theorem}

\begin{proof}
    Suppose for some finite graph $H$ that every $(H,K_3)$-free graph is $k$-colourable. 
    Let $G$ be a $(H+\ell P_1, K_3)$-free graph.
    Note if $\ell=0$ then $G$ is $(H, K_3)$-free and hence $k$-colourable.
    So suppose $\ell \geq 1$
    Let $S$ be an independent set of size at most $\ell \geq 1$ such that either $|S|=\ell$ or, if no such indepedent set exists, $|S|$ is maximal. 
    It suffices to show that $G$ is $(|S|+k)$-colourable.
    Label the vertices such that $S=\{v_1,v_2, \ldots, v_{|S|}\}$.
    For each $i \in \{1, 2, \ldots |S|\}$ let $S_i$ be the vertices adjacent to $v_i$ in $S$ but not adjacent to $v_j$ for any $j<i$.
    Note $S$ together with each $S_i$ partitions $N[S]$.
    Moreover as $G$ is $K_3$-free then each $S_i$ is an independent set.
    Thus we can assign $S$ and each $S_i$ their own colour to colour $N[S]$ with $|S|+1$ colours.
    Now consider $G-N[S]$.
    Note that if $|S|$ is maximal then $G-N[S]$ is empty so $G$ is $(|S|+1)$-colourable and hence $(\ell+k)$-colourable as $|S|<\ell$ and $1\le k$.
    Thus we may assume $|S|=\ell$ in which case $G-N[S]$ is $H$-free.
    Therefore $G-N[S]$ is $(H,K_3)$-free and hence $k$-colourable
    We can reuse the the colour assigned to $S$ and colour $G$ with $|S|+k$ colours.
\end{proof}

Thus, immediate corollaries of this are that every $(H+\ell P_1,K_3)$-free graph is $(\ell+4)$-colourable when $H$ is $2P_3$ or $P_4+P_2$, as the $4$-colourability of $(H,K_3)$-free graphs for each $H$ was shown in \cite{PYATKIN2013715} and \cite{ChenWuZhang025}, respectively. Further, from the result that every $(P_t,K_3)$-free graph is $(t-2)$-colourable~\cite{RanderathSchiermeyer2004} we have the following corollary.

\begin{corollary} \label{cor:P_4+lP_1_K_3free}
    For all $\ell\ge 0$ and $t\ge 4$, every $(P_t+\ell P_1, K_3)$-free graph is $\left(t+\ell-2\right)$-colourable.
\end{corollary}

It was shown in \cite{GRAVIER2003} that $\chi(G) \leq (t-2)^{\omega(G)-1}$ for $P_t$-free graphs. This $\chi$-bound implies that every $(P_4+\ell P_1, K_k)$-free graph is $(2\ell +2)^{k-2}-$colourable. We improve this in the next result by generalizing Corollary \ref{cor:P_4+lP_1_K_3free} for $t=4$.

\begin{theorem}
    For all $\ell\ge 0$ and $k\ge 3$, every $(P_4+\ell P_1, K_k)$-free graph $G$ has 

    $$\chi(G) \leq \ell^{k-2}+2\ell^{k-3}+3\ell^{k-4}\cdots+(k-2)\ell+(k-1). $$

\end{theorem}

\begin{proof}
    We proceed by induction on $k$.
    The case of $k=3$ follows from Corollary \ref{cor:P_4+lP_1_K_3free}.
    So suppose for all $\ell\ge 0$ and some $k\ge 3$ that every graph is $f(k)$-colourable where $f(k)=\ell^{k-2}+2\ell^{k-3}+\cdots+(k-1)$. 
    Let $G$ be a $(P_4+\ell P_1, K_{k+1})$-free graph.
    Note if $\ell=0$ then $G$ is $P_4$-free and hence perfect.
    Therefore  $\chi(G)=\omega(G) \leq k$ as $G$ is $K_{k+1}$-free.
    Let $S$ be an independent set of size at most $\ell \geq 1$ such that either $|S|=\ell$ or $S$ is maximal.
    It suffices to show that $G$ is $(|S|\cdot f(k)+k)$-colourable.
    Label the vertices such that $S=\{v_1,v_2, \ldots, v_{|S|}\}$.
    Moreover, for each $i \in \{1, 2, \ldots |S|\}$, let $S_i$ be the vertices adjacent to $v_i$ in $S$ but not adjacent to $v_j$ for any $j<i$.
    Note $S$ together with each $S_i$ partition $N[S]$.
    Since $G$ is $K_{k+1}$-free, each $S_i$ must be $K_{k}$-free and hence $(P_4+\ell P_1, K_{k})$-free graph.
    Thus, by induction, each $S_i$ is $f(k)$-colourable.
    So we can colour $N[S]$ with $|S|\cdot f(k)+1$ colours by also assigning $S$ its own colour.
    Now consider $G-N[S]$.
    Note that if $S$ is maximal then $G-N[S]$ is empty so $G$ is $(|S|\cdot f(k)+1)$-colourable.
    Thus we may assume $|S|=\ell$ in which case $G-N[S]$ is $P_4$-free.
    Therefore $G-N[S]$ is perfect and $\chi(G-N[S])=\omega(G-N[S]) \leq k$ as $G$ is $K_{k+1}$-free.
    We can reuse the the colour assigned to $S$ and colour $G$ with $|S|\cdot f(k)+k$ colours.
\end{proof}


\section{Conclusion}\label{sec:conclusion}

In this work we gave the first results on the number of vertex-critical graphs in subfamilies of $(P_4+\ell P_1)$-free graphs with no restrictions on $\ell$.

A natural open question is if it is possible to remove the restriction of either $\broom{4}{m}$ or $\broom{3}{m}^{+}$ from Theorem~\ref{thm:P4UellP1F+mBm}. 
Removing the former would result in the highly desirable corollary that $\numcrit{k}{P_4+\ell P_1,P_5}<\infty$ for all $k\ge 1$ and $\ell\ge 0$.
Removing the latter would have, as its most specific new corollary, that $\numcrit{k}{P_4+\ell P_1,K_3}<\infty$ which was discussed in Section \ref{sec:K_k}.  
We think this is a particularly interesting open problem, especially given the very recent results that all $(P_5+P_1,K_3)$-free graphs are $3$-colourable and $\numcrit{4}{P_{11},K_3}=\infty$~\cite{zhouetal2025P11C3freecritical}.

\begin{problem}
    Determine if $\numcrit{k}{P_4+\ell P_1,K_3}$ is finite for all $\ell\ge 2$.
\end{problem}

Beyond the number of vertex-critical graphs in subfamilies, the colourability in general is often of interest. Our results in Section \ref{sec:K_k} imply that every $(P_4+\ell P_1, K_k)$-free graph is $\mathcal{O}(\ell^{k-2})$-colourable. However, a conjecture in \cite{trotignon2018chi} posits that $(P_t, K_k)$-free graphs are polynomial-bounded by $k$. Which leads to the following analogous problem.

\begin{problem}
    For each $k\geq 3$, determine the precise value $M_k$ such that $\chi(G)\le M_k$ for all $(P_4+\ell P_1,K_k)$-free graphs $G$.
\end{problem}

Another direction for future research would be to use our results to prove the finiteness of $k$-vertex-critical $(P_5,H)$-free graphs when $H$ is any of chair, bull, or cricket. 
The idea would be that from our results, it is known that there are only finitely many that are $(P_4+\ell P_1)$-free, so it remains only to prove that there are only finitely many that contain an induced $P_4+\ell P_1$. 
This could be a promising technique as $\ell$ can be made as large as possible as long as it only depends on $k$, and thereby giving lots of control to the interested researcher.


\begin{credits}
\subsubsection{\ackname} 
The reseach of Ben Cameron was supported by the Natural Sciences and Engineering Research Council of Canada (NSERC), grants RGPIN-2022-03697 and DGECR-2022-00446. The research of Iain Beaton was also supported by NSERC grants RGPIN-2025-06012 and DGECR-2025-00001
\subsubsection{\discintname}
The author(s) have no competing interests to declare that are relevant to the content of this article.
\end{credits}
%
%

\bibliographystyle{splncs04}
\bibliography{refs}





\end{document}